\documentclass{amsart}
\usepackage[showeqnr,theoremdefs]{allergy}

\usepackage[inline]{enumitem}

\usepackage{tikz}
\usetikzlibrary{decorations.pathreplacing,shapes,arrows,positioning}
\tikzset{
	vertex/.style = {
		circle,
		fill            = black,
		outer sep = 2pt,
		inner sep = 1pt,
	}
}
\usetikzlibrary{spy}
\usetikzlibrary{matrix}

\usetikzlibrary{matrix}

\usepackage{siunitx} 

\usepackage{pdfsync}
\graphicspath{{fig/}}

\usepackage{subcaption}
\usepackage{pdfsync}

\usepackage{enumitem}
\usepackage[T1]{fontenc} 

\newcommand*\R{\textrm{I\kern-0.21emR}}
\newcommand*\RR{\mathbb{R}}
\newcommand*\N{\textrm{I\kern-0.21emN}}
\newcommand*\NN{\mathbb{N}}

\newcommand*\suppdata[1][\data]{\operatorname{supp}(#1)}
\newcommand {\MM} { {\mathcal{M}} }
\newcommand*\imcone{A(\MM_+)}
\newcommand*\one{\boldsymbol{1}}

\newcommand*\iz{i_0}

\newcommand*\loss{\ell}
\newcommand*\ndet{m}

\renewcommand{\geq}{\geqslant}
\renewcommand{\leq}{\leqslant}

\renewcommand{\geq}{\geqslant}
\renewcommand{\leq}{\leqslant}

\newcommand\numberthis{\addtocounter{equation}{1}\tag{\theequation}}

\newcommand {\Chi} {{\bf \raise 2pt \hbox{$\chi$}} }

\newcommand*\Cont{\mathcal{C}}
\newcommand{\scl}[2]{\pairing*{#1}{#2}}

\newcommand {\PP} { {\mathbb P} }

\newcommand*\dd{\mathrm{d}}

\newcommand{\data}{y}

\newcommand{\beq}{\begin{equation}}
\newcommand{\eeq}{\end{equation}}
\newcommand{\bea} {\begin{array}{rl}}
\newcommand{\eea} {\end{array}}
\newcommand{\bepa}{\left\{ \begin{array}{l}}
\newcommand{\eepa} {\end{array}\right.}
\newcommand{\bmu}{\begin{multline}}
\newcommand{\emu}{\end{multline}}
\DeclareMathOperator*{\argmin}{arg\,min}
\DeclareMathOperator*{\argmax}{arg\,max}

\title{The ML-EM algorithm in continuum: \\sparse measure solutions}

\date{\today}
\author{Camille Pouchol}
\address{Department of Mathematics, KTH Royal Institute of Technology, 100 44 Stockholm, Sweden.}
\email{pouchol@kth.se}

\author{Olivier Verdier}
\address{Department of Mathematics, KTH Royal Institute of Technology, 100 44 Stockholm, Sweden.}
\email{olivierv@kth.se}
\address{Department of Computing, Electrical Engineering and Mathematical Sciences, Western Norway University of Applied Sciences, Bergen, Norway.}
\email{olivier.verdier@hvl.no}

\begin{document}

\newcounter{assum}

\maketitle
\begin{abstract}
  Linear inverse problems $A \mu = \data$ with Poisson noise and non-negative unknown $\mu \geq 0$ are ubiquitous in applications, for instance in Positron Emission Tomography (PET) in medical imaging.
  The associated maximum likelihood problem is routinely solved using an expectation-maximisation algorithm (ML-EM).
  This typically results in images which look spiky, even with early stopping.
  We give an explanation for this phenomenon.
  We first regard the image $\mu$ as a measure.
  We prove that if the measurements $\data$ are not in the cone $\{A \mu, \mu \geq 0\}$, which is typical of low injected dose, likelihood maximisers must be sparse, i.e., typically a sum of point masses. 
We also show a weak sparsity result for cluster points of ML-EM.  On the other hand, in the low noise regime, we prove that cluster points of ML-EM are optimal measures with full support.
  Finally, we provide concentration bounds for the probability to be in the sparse case, and a set of numerical experiments supporting our claims.

\end{abstract}

\section{Introduction}
\label{sec_intro}
In various imaging modalities, recovering the image from acquired data can be recast as solving an inverse problem of the form
$A \mu = \data$, where $A$ is a linear operator, $\data$ represents noisy measurements and $\mu$ the image, with $\mu \geq 0$ usually a desirable property. 
The problem thus becomes $\min_{\mu \geq 0} d(\data, A \mu)$ where $d$ is some given distance or divergence.

When the model is finite-dimensional, the operator $A$ is simply a matrix $A = (a_{ij}) \in \mathbb{R}^{\ndet \times r}$.
If we assume a Poisson noise model, i.e., $\data_i \sim \mathcal{P}((A \mu)_i)$ with independent draws, 
the corresponding (negative log) likelihood problem
is equivalent to
\begin{equation}
  \label{discrete_likelihood}
  \min_{\mu \geq 0} \quad
  d(\data || A \mu)
,
\end{equation}
where $d$ is the Kullback--Leibler divergence.
As it turns out, this statistical model
is similar to the familiar
non-negative least-squares regression corresponding to Gaussian noise, but for a different distance functional: if $\data \notin \{A \mu, \mu \geq 0\}$, it is projected onto it in the sense of of the divergence $d$, whereas if it belongs to this image set, any $\mu \geq 0$ such that $A \mu = \data$ will be optimal.
  
The celebrated Maximum Likelihood Expectation Maximisation algorithm (ML-EM) precisely aims at solving~\eqref{discrete_likelihood} and was introduced by Shepp and Vardi~\cite{Shepp1982, Vardi1985}, in the particular context of the imaging modality called Positron Emission Tomography (PET). It was proposed earlier in another context and is as such often called the Richardson--Lucy algorithm~\cite{Richardson1972, Lucy1974}. 

\noindent
The ML-EM algorithm is iterative and writes
\beq
\label{discrete_mlem}
\mu_{k+1} = \frac{\mu_{k}}{A^T 1}\,  A^T\bigg(\frac{\data}{A \mu_k}\bigg),
\eeq
starting from $\mu_0 > 0$, usually $\mu_0 = 1$.

This algorithm is an expectation-maximisation (EM) algorithm, and as such it has many desirable properties: it preserves non-negativity and the negative log-likelihood decreases along iterates~\cite{Dempster1977}. It can also be interpreted in several other ways~\cite{Vardi1985, Csiszar1984, Benvenuto2014}, see~\cite{Natterer2001} for an overview and~\cite{Qi2006} for related iterative algorithms.
The expectation-maximisation point of view has also led to alternative algorithms~\cite{Fessler1993}, but in spite of various competing approaches, ML-EM (actually, its more numerically efficient variation OSEM~\cite{Hudson1994}) has remained the algorithm used in practice in many PET scanners.

Despite its success, the ML-EM algorithm \eqref{discrete_mlem} is known to produce undesirable spikes along the iterates,
where some pixels take increasingly high values. 
The toy example presented in \autoref{example_intro} is an example of such artefacts in the case of PET. The reconstruction of a torus of maximum $1$ with $100$ iterations of ML-EM indeed exhibits some pixels with values as high as about $6$. 

This phenomenon has long been noticed in the literature, where images
are referred to as ``spiky'' or ``speckled''~\cite{Silverman1990},
others talking about ``the chequerboard effect''~\cite{Vardi1985}. In the discrete case, the works~\cite{Byrne1993, Byrne1995, Byrne1996} have provided a partial explanation for this result. The author proves that, under general conditions (which include $\ndet<r$), the minimum of~\eqref{discrete_likelihood} is such that it has at most $\ndet-1$ non-zero entries whenever $\data \notin \{A \mu, \mu \geq 0\}$.  

To the best of our knowledge, a theoretical justification for the subsistence of \textit{only a few} non-zero entries has however remained elusive. 

\begin{figure*}
		\centering	
		\begin{subfigure}[t]{0.42\textwidth}
			\centering	
			\begin{tikzpicture}[
			remember picture,
			spy using outlines={%
				circle,
				red,
				magnification=3,
				size=2cm,
				connect spies,
				spy connection path={\draw[thick] (tikzspyonnode) -- (tikzspyinnode);}
			}
			]
			\node {\includegraphics[width=\textwidth, clip]{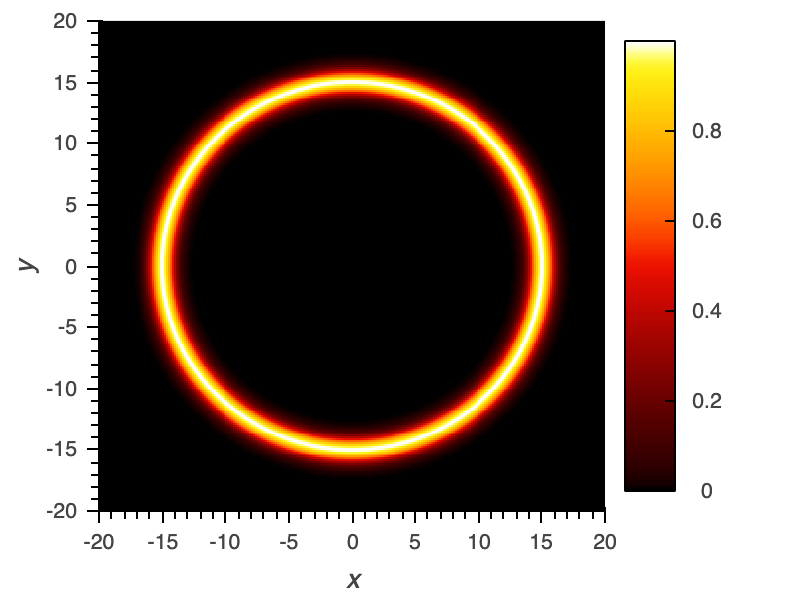}};
			\end{tikzpicture}
		\end{subfigure}
		\begin{subfigure}[t]{0.42\textwidth}
			\centering		
			\begin{tikzpicture}[
			remember picture,
			spy using outlines={%
				circle,
				blue,
				magnification=3,
				size=1cm,
				connect spies,
				spy connection path={\draw[thick] (tikzspyonnode) -- (tikzspyinnode);}
			}
			]
			\node {\includegraphics[width=\textwidth, clip]{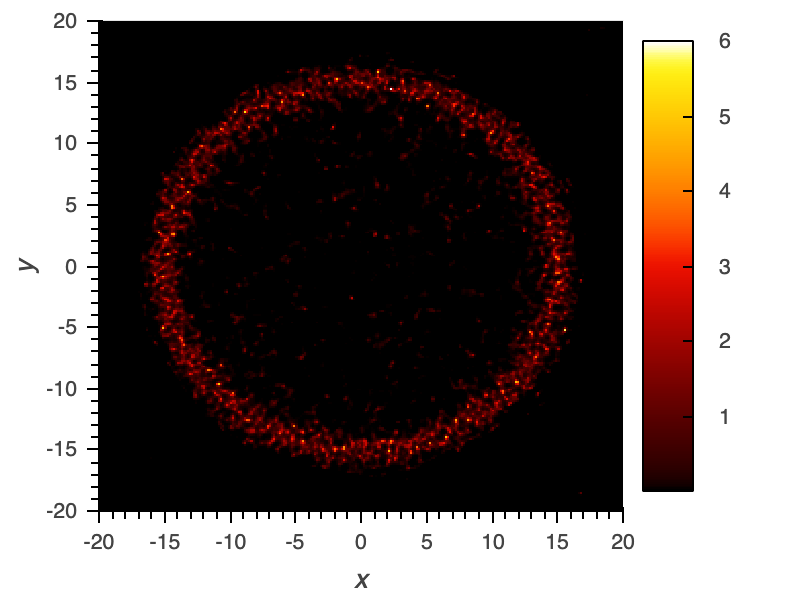}};
			\spy on (-0.05,1.40) in node [left] at (1.2,3.1);
			\end{tikzpicture}
		\end{subfigure}
		\caption{
      Phantom and reconstruction after $100$ iterations of ML-EM, with a zoom on the region containing the pixel of highest value.
		}
		\label{example_intro}
	\end{figure*}



The aim of the present paper is to better understand that phenomenon via the analysis \emph{in a continuous setting} of the minimisation problem~\eqref{discrete_likelihood} and the corresponding ML-EM algorithm~\eqref{discrete_mlem}.
The continuous setting here refers to the image not being discretised on a grid.
Note however that we keep the data space discrete.

Informally, considering $\mu$ as an element in some function space, we consider forward operators $A$ of the form
\[ (A \mu)_i := \int_K a_i(x) \mu(x) \, \dd x,\]
where $K$ is the compact on which one aims at reconstructing the image, and $a_i$ is some non-negative function on $K$.
This covers a wide range of applications, including PET.



One of our motivations
is to derive algorithms for Poisson inverse problems with movement, for example for PET acquisition of a moving organ~\cite{Jacobson2003}.
In that case, movement can be modelled by deformations of images which do not easily carry over to discretised images (simply because interesting deformations do not preserve the grid).
It is then desirable to express the problem in a continuous setting, in order to both analyse the algorithms proposed in the literature, and to derive new ones~\cite{Hinkle2012, Oktem2019}.



The field of inverse problems for imaging, with a continuum description of the unknown image, is abundant~\cite{Bertero1998, Baumeister2005}.
Most often, the image is taken to be a function in some appropriate Sobolev space.
To the best of our knowledge, however, there are relatively few results concerning the continuum description of the Poisson likelihood and the ML-EM algorithm for solving it.

In~\cite{Multhei1987, Multhei1989, Multhei1992} and~\cite{Resmerita2007}, both the image and data are considered in continuum,
with a deterministic description of noise.
These authors assume that detectors $a_i$ lie in $L^{\infty}(K)$ and correspondingly assume that the image $\mu$ lies in $L^1(K)$.
They study the convergence properties of the corresponding ML-EM algorithm in detail.
In the first series of three papers, the compact $K$ is restricted to $K = [0,1]$. 

Our paper differs from these works in that we do not make the two following restrictive assumptions, common to \cite{Multhei1987, Multhei1989, Multhei1992,Resmerita2007}.
The first restriction is to assume the existence of a non-negative solution $\mu$ to the equation $A \mu = \data$, assumed to lie in $L^1(K)$.
The second restriction is to assume that the functions $a_i$ are bounded away from zero.
This last assumption is unrealistic for some applications such as PET~\cite[Remark~6.1]{Resmerita2007}.

The seminal paper \cite{Mair1996}
considers the optimisation problem over the set of non-negative Borel measures as we do.
They obtain the corresponding likelihood function informally as the limit of the discrete one,
but do not prove that it is an actual maximum likelihood problem for the PET statistical model.
They then proceed to study the problem of whether minimisers can be identified with bounded functions,
and not merely measures which might have a singular part.
They indeed note that in some very specific cases (see also~\cite{Multhei1992}), one can prove that the minimiser should be a Dirac mass.
They speculate that there might be a link with the usual spiky results from ML-EM.
They, however, do not provide any general conditions for sparsity.


Working in the space of non-negative measures $\MM_+$, our main contributions are as follows:

\begin{description}
\item[Continuous framework]
  We prove
  that the continuous setting of measures is precisely the maximum likelihood problem with a Poisson point process model (\autoref{pet_stat_model}),
  and that the natural generalisation of the ML-EM iterates~\eqref{eq:mlemcont} indeed corresponds to the expectation-maximisation method associated to that continuous statistical model (see \autoref{sec:statmodel}).
\item[Sparsity] We give a precise sparsity criterion (sparsity means that any optimal solution has singular support): if the data $\data$ is outside the cone $\imcone$, then all optimal solutions are necessarily sparse (\autoref{cor_sparse});
  if the data $\data$ is inside the cone $\imcone$, then there exist absolutely continuous solutions  (\autoref{smooth_int}).
\item[Properties of ML-EM iterates] We show the expected properties of the ML-EM iterates, namely monotonicity (\autoref{cor_dist}) and the fact that cluster points are fixed points of the algorithm (\autoref{Fixed_point}). 
\item[Properties of ML-EM solutions] 
  We show that in the non-sparse case, i.e., when an absolutely continuous solution exists as just mentioned, ML-EM iterates are optimal and have full support (\autoref{thm_smooth}). In the sparse case, we provide a weak sparsity result for cluster points (\autoref{weak_sparsity}), and we give an explicit example of ML-EM converging to a sum of point masses (\autoref{Dirac}). 
\item[Effect of noise] We derive estimates on the probability to be in the sparse case, depending on the noise level (\autoref{prop:nbound}, \autoref{exp_dec}).
\item["Spiky" artefacts] With these results, we provide an explanation for the artefacts of \autoref{example_intro}: they are related to the sparsity result.
By weak duality, we can indeed \emph{certify} that 
optimal measures 
should be sums of point masses in that case, as detailed in \autoref{sec:numerics} dedicated to simulations.
\end{description}

\subsection*{Outline of the paper}
The paper is organised as follows.
In \autoref{sec_def}, we introduce the functional and ML-EM in continuum in detail, with all the necessary notations, normalisations, definitions and useful properties about Kullback--Leibler divergences.
\autoref{sec_min} contains all results on the functional minimisers, starting from the optimality conditions to the diverging cases of the data $\data$ being inside or outside the cone $\imcone$.
\autoref{sec_mlem} is devoted to the algorithm ML-EM itself, with the proof of its usual properties in continuum together with the implications they have on the case where the data $\data$ is in the cone $\imcone$.
In \autoref{sec_stat}, we estimate the probability that the data $\data$ ends up outside the image cone $\imcone$.
In \autoref{sec:numerics}, we present simulations which confirm our theoretical predictions.
Finally, in \autoref{sec_conc} we conclude with open questions and perspectives.


\section{Maximum likelihood and ML-EM in continuum}
\label{sec_def}

\subsection{Mathematical background}

\subsubsection{Space of Radon measures.}
As stated in the introduction, we model the image to reconstruct as a non-negative measure $\mu$ defined on a compact set $K$.
Some of our results require $K\subset\mathbb{R}^p$ (typically, $p=2$ or $3$).

More precisely, we will consider the set of Radon measures, denoted $\MM(K)$ and defined as the topological dual of the set of continuous functions over $K$, denoted $\Cont(K)$.
The space of non-negative measures will be denoted by $\MM_+(K)$.
For brevity, we will often write $\MM$ for $\MM(K)$ and $\MM_+$ for $\MM_+(K)$ when there is no ambiguity as to the underlying compact $K$. 

We identify a linear functional $\mu \in \MM_+$ with its corresponding Borel measure (as per the Riesz--Markov representation Theorem), using $\mu(B)$ to denote the measure of a Borel subset $B$ of $K$. We will also sometimes write the dual pairing between a measure $\mu \in \MM$ and a function $f \in \Cont(K)$ as
\[\scl{\mu}{f} = \int_K f d\mu.
\]
The support of a measure $\mu \in \MM$ is defined as the closed set
\[\operatorname{supp}(\mu) := \setc[\big]{x \in K}{ \mu(N) > 0, \quad \forall N \in N(x)}
  ,
\]
where $N(x)$ is the set of all open neighbourhoods of $x$. 

Finally, recall that, by the Banach--Alaoglu Theorem, bounded sets in $\MM$ are weak-$\ast$ compact~\cite{Rudin1991}.

\subsubsection{Kullback--Leibler divergence.}
We here recall the definition of the Kullback--Leibler (KL) divergence. Instead of giving the general definition, we make the two instances that will actually be needed in this paper explicit, for (non-normalised) non-negative vectors in $\mathbb{R}^{\ndet}$, and for probability measures on $K$.

\begin{itemize}
\item 
\emph{For vectors in $\mathbb{R}^{\ndet}$.}
For any two non-negative vectors $u$ and $v$ in $\mathbb{R}^{\ndet}$, we define the Kullback--Leibler divergence between $u$ and $v$ as
\[d(u||v) := \sum_{i=1}^{\ndet} \bigg(v_i - u_i - u_i \log\bigg(\frac{v_i}{u_i}\bigg)\bigg), \]
with the convention $0 \log(0) = 0$ and $d(u||v) = +\infty$ if there exists $i$ such that $u_i = 0$ and $v_i>0$.

\item 
\emph{For probability measures on $K$.}
For any two probability measures $\mu$ and $\nu$ on $K$, we define the Kullback--Leibler divergence between $\mu$ and~$\nu$ 
\[D(\mu|| \nu) :=  \int_K\log\bigg(\frac{\dd \mu}{\dd \nu}\bigg) \, d\mu, \]
if $\mu$ is absolutely continuous with respect to $\nu$ (denoted $\mu \ll \nu$) and $\log\big(\frac{\dd \mu}{\dd \nu}\big)$ is integrable with respect to $\mu$.
Here $\frac{\dd \mu}{\dd \nu}$ stands for the Radon--Nikodym derivative of $\mu$ with respect to $\nu$.
Otherwise, we define $D(\mu|| \nu) := +\infty$.
\end{itemize}

When a measure is absolutely continuous with respect to the Lebesgue measure on~$K \subset \RR^p$, we simply say that it is absolutely continuous. Any reference to the Lebesgue measure implicitly assumes that $K$ stands for the closure of some bounded domain in $\RR^p$ (i.e., a bounded, connected and open subset of~$\RR^p$).

\subsection{Statistical model}

\label{sec:statmodel}

We want to recover a measure $\mu \in \MM_+$ from independent Poisson distributed measurements 
\beq
\label{stat_model}
N_i \sim \mathcal{P}\bigg(\int_K a_i d \mu \bigg), \qquad  i =1, \ldots, \ndet,
\eeq
with
\beq
a_i \geq 0, \quad a_i \in \Cont(K), \qquad i =1, \ldots, \ndet.
\eeq

\subsubsection{Positron Emission Tomography~\cite{Ollinger1997}.} In PET, a radiotracer injected into the patient and, once concentrated into tissues, disintegrates by emitting a positron. This process is well known to be accurately modelled by a Poisson point process, itself defined by a non-negative measure. After a short travel distance called \emph{positron range}, this positron interacts with an electron. The result is the emission of two photons in random opposite directions. Pairs of detectors around the body then detect simultaneous photons, and the data is given by the number of counts per pair of detectors.

In the case of PET, $\ndet$ is the number of of detectors (i.e., pairs of {single} detectors).
For a given point~$x \in K$ and detector $i \in \{1, \ldots, \ndet\}$, $a_i(x)$ then denotes the probability that a positron emitted in~$x$ will be detected by detector $i$.  

Finally, we will throughout the paper assume  
\beq
\label{positive}
\sum_{i=1}^\ndet a_i > 0 \text{ on } K.
\eeq
For PET, this amounts to assuming that that the points in $K$ are in the so-called field of view, namely that any emission has a non-zero probability to be detected.

\subsubsection{Derivation of the statistical model~\eqref{stat_model} for PET}
We proceed to give a proof that the statistical model~\eqref{stat_model} (and thus, the corresponding likelihood function) applies to PET. Here, we assume that the emission process of PET is modelled by a Poisson point process, defined by a measure $\mu \in \MM_+$, and that each point drawn from the Poisson process has a probability $a_i(x)$ to be detected by detector $i$. 

\begin{proposition}
\label{pet_stat_model}
The statistical model~\eqref{stat_model} applies to PET. 
\end{proposition}

\begin{proof}
The proof relies on the following properties of Poisson point processes~\cite{Last2017}:
\begin{itemize}
\item \emph{law of numbers}: the number of points emitted by a Poisson process of intensity $\mu$ follows the Poisson law with parameter $\int_K \ndet \mu = \mu(K)$.
\item \emph{thinning property}: the points that are kept with (measurable) probability $p : K \mapsto [0,1]$ still form a Poisson point process, with intensity $p \mu$, and it is independent from that of points that are not kept. This property generalises to $p_i$, $1 \leq i \leq \ndet$ with $\sum_{i=1}^\ndet p_i(x) = 1$ for all $x \in K$. 
\end{itemize}

By the thinning property, the families of points which lead to an emission detected in detector $i$, $i = 1, \ldots, \ndet$, are all independent Poisson processes with associated measure $a_i \mu$, for $i=1, \ldots, \ndet$. Thus, the random variables $N_i$ representing the number of points detected in detector $i$ are independent and of law $\mathcal{P}(\int_K a_i \ndet \mu)$, which proves the claim.
\end{proof}

\subsubsection{Maximum likelihood problem.}

The likelihood corresponding to the statistical model~\eqref{stat_model} reads
\[L(N_1, \ldots, N_p ; \mu)  =  \prod_{i=1}^\ndet L(N_i ; \mu) = \prod_{i=1}^{\ndet} e^{-\int_K a_i \dd\mu } \frac{(-\int_K a_i \dd\mu )^{N_i}}{N_i!}, \]
since $\PP(N_i = n_i) = e^{-\int_K a_i \dd\mu} \frac{(-\int_K a_i \dd\mu)^{n_i}}{n_i!}$. 

Dropping the factorial terms (they do not depend on $\mu$ and will thus play no role when maximising the likelihood),
we get
\beq
\log(L(N_1, \ldots, N_p ; \mu))  = -\sum_{i=1}^{\ndet} \int_K a_i \dd\mu +\sum_{i=1}^{\ndet} N_i \log\left(\int_K a_i \dd \mu \right).
\eeq 

The corresponding maximum likelihood problem, written for a realisation $n_i$ of the random variable $N_i$, $i =1, \ldots, \ndet$, is given by 
\beq
\label{likelihood}
\operatorname*{max}_{\mu \in \MM_+} \quad- \int_K \bigg(\sum_{i=1}^{\ndet} a_i\bigg) \dd\mu + \sum_{i=1}^{\ndet} n_i \log\left(\int_K a_i \dd\mu\right).
\eeq
Defining the operator
\begin{align*}
A \colon   \MM &\longrightarrow   \mathbb{R}^{\ndet} \\
 \mu & \longmapsto \left(\scl{\mu}{a_i}\right)_{1 \leq i \leq \ndet}, 
\end{align*}
the optimisation problem conveniently rewrites in terms of the Kullback--Leibler divergence: upon adding constants and taking the negative log-likelihood problem, it reads
\beq
\label{likelihood_distance}
\operatorname*{min}_{\mu \in \MM_+} \quad d(n || A \mu). 
\eeq

\subsubsection{ML-EM iterates.}
We now define the ML-EM algorithm, which aims at solving the optimisation problem~\eqref{likelihood}. It is given by the iterates
\beq
\mu_{k+1} =  \frac{ \mu_k}{ \sum_{i=1}^{\ndet} a_i } \left(\sum_{i=1}^{\ndet}\frac{ n_i a_i}{\int_K a_i \dd \mu_k }\right),
\eeq
starting from an initial guess $\mu_0 \in \MM_+$. In agreement with the Kullback--Leibler divergence, we choose the convention that divisions of the form $0/0$ are of course taken to be equal to $0$.

Note that this algorithm can be shown to be an EM algorithm for the continuous problem. The proof is beyond the scope of this paper, so we decide to omit it, but we just mention that the corresponding so-called \emph{complete data} would be
 given by the positions of points together with the detector that has detected each of them.
 
\subsection{Normalisations}
Due to the assumption \eqref{positive}, we may without loss of generality assume that
\begin{equation}
  \label{eq:normalisation}
  \sum_{i=1}^{\ndet} a_i = 1
\end{equation}
on $K$.
Otherwise we could just define
$\tilde{\mu} = (\sum_{i=1}^{\ndet} a_i) \mu$ and $\tilde{a}_i = a_i / (\sum_{j=1}^{\ndet} a_j)$.
This normalisation now implies $0 \leq a_i \leq 1$ for all $i = 1, \ldots, {\ndet}$.

We further normalise the measures by dividing the functional by $n := \sum_{i=1}^{\ndet} n_i$, considering $\mu:= \frac{\mu}{n}$ to remove the factor. 
We then define 
\[\data_i :=  \frac{n_i}{n}.\]
From now on, we consider the optimisation problem (minimisation of the negative log-likelihood):
\beq
\label{eq:optprob}
\min_{\mu \in \MM_+} \quad \loss(\mu),
\eeq 
where
\begin{align}
  \label{eq:defloss}
\loss(\mu)& := \int_K \dd\mu  - \sum_{i=1}^{\ndet} \data_i \log\left(\int_K a_i \dd\mu\right) \\
& = \scl{\mu}{1}  - \sum_{i=1}^{\ndet} \data_i \log\left(\scl{\mu}{a_i}\right),
\end{align}
defined to be $+\infty$  for any measure such that $\scl{\mu}{a_i} =  0$ for some $i \in \suppdata$, where
\[
  \suppdata \coloneqq \setc[\big]{i = 1, \dotsc, \ndet}{\data_i > 0}
.
\]

\noindent
After normalisation, the ML-EM iterates are given by 
\beq
\label{eq:mlemiterate}
\mu_{k+1} = \mu_k \left(\sum_{i=1}^{\ndet}\frac{ \data_i a_i}{\int_K a_i \dd\mu_k }\right) = \mu_k \left(\sum_{i=1}^{\ndet}\frac{ \data_i a_i}{\scl{\mu_k}{a_i} }\right) .
\eeq
We recall the property that ML-EM preserves the total number of counts: $\scl{\mu_k}{1} = 1$ for all $k \geq 1$, which corresponds to $\scl{\mu_k}{1}= n = \sum_{i=1}^{\ndet} n_i$ before normalisation. We also emphasise the important property that iterations cannot increase the support of the measure, namely 
\[\forall k \in \mathbb{N}, \; \operatorname{supp}(\mu_{k+1}) \subset \operatorname{supp}(\mu_k).\]

\subsubsection*{ML-EM iterates are well-defined.}
We assume throughout that the initial measure $\mu_0$ fulfils
\beq
\label{init_1}
 \scl{\mu_0}{a_i} >  0 \qquad \forall i \in \suppdata
.
\eeq
Note that usual practice is to take $\mu_0$ to be absolutely continuous with respect to the Lebesgue measure, typically $\mu_0 = 1$, in which case \eqref{init_1} is satisfied. 

The following simple Lemma shows that assumption~\eqref{init_1} ensures that the iterates are well-defined.

\begin{lemma}
The ML-EM iterates~\eqref{eq:mlemiterate} satisfy 
\[  \scl{\mu_k}{a_i} >  0 \implies \scl{\mu_{k+1}}{a_i} > 0
  \qquad
  i \in \suppdata
  .\]
\end{lemma}

\begin{proof}
From the Cauchy--Schwarz inequality, $\mu_k(K) \scl{\mu_k}{a_i^2} \geq \scl{\mu_k}{a_i}^2$.  Combined with the definition of ML-EM iterates, this entails for any $i \in \suppdata$,
\[\scl{\mu_{k+1}}{a_i} = \sum_{j=1}^{\ndet} \data_j \,  \frac{\scl{\mu_k}{a_i a_j}}{\scl{\mu_k}{a_j}} \geq \data_i \,  \frac{\scl{\mu_k}{a_i^2}}{\scl{\mu_k}{a_i}}  \geq \data_i \frac{\scl{\mu_k}{a_i}}{\mu_k(K)}  > 0.\]
\end{proof}


\subsection{Adjoint and cone}
\label{sec:coneadjoint}

Since $A \colon \Cont(K)^* \to \mathbb{R}^{\ndet}$, we can define its adjoint $A^\ast\colon \RR^{\ndet} \to \Cont(K)$ (identifying $\RR^{\ndet}$ as a Euclidean space with its dual), which is given by
\begin{equation}
  \label{eq:defadjoint}
  A^\ast w = \sum_{i=1}^{\ndet} w_i a_i,
    \qquad w \in \RR^{\ndet}
  .
\end{equation}

The set $\imcone =\left\{ A \mu, \; \mu \in \MM_+\right\} \subset\mathbb{R}^{\ndet}$ is a closed and convex cone and, as proved in~\cite{Georgiou2005}, its dual cone $\imcone^\ast$ can be characterised as being given by the set of vectors $\lambda \in \mathbb{R}^{\ndet}$ such that $\sum_{i=1}^{\ndet} \lambda_i a_i \geq 0$ on $K$, i.e.,
\begin{equation}
  \label{eq:propdualcone}
  \imcone^\ast = \setc*{\lambda \in \mathbb{R}^{\ndet}}{A^* \lambda \geq 0 \text{ on } K}
  .
\end{equation}
As a result, the interior of the dual cone $\imcone^{\ast}$ is given by the vectors $\lambda \in \mathbb{R}^{\ndet}$ such that $A^* \lambda > 0$ on $K$.

The normalisation condition \eqref{eq:normalisation} can now be rewritten
\begin{equation}
  \label{eq:normalisationop}
  A^* \one = \one
  ,
\end{equation}
where \(\one\) is the vector of $\RR^{\ndet}$ which all components are one: \(\one = (1,\dotsc,1)\).
Moreover, we can rewrite the ML-EM iteration \eqref{eq:mlemiterate} as
\[
  \mu_{k+1} = \mu_k \, A^*\Big(\frac{\data}{A \mu_k}\Big)
 ,
  \]
  which is the continuous analogue to the discrete case~\eqref{discrete_mlem},
  taking into account the normalisation \eqref{eq:normalisationop}.

\subsubsection*{Minimisation over the cone. }
The problem $\operatorname*{min}_{\mu \in \MM_+} d(\data || A\mu)$, is equivalent to the following minimisation problem over the cone $\imcone$:
\beq
\label{likelihood_cone}
\min_{w \in \imcone} \quad d(\data || w).
\eeq
Indeed, if $w^\star$ is optimal for the problem~\eqref{likelihood_cone}, any $\mu^\star$ such that $A\mu^\star = w^\star$ is optimal for the original problem.
From the property $d(\data || w) = 0 \iff \data = w$, we also infer that when $\data \in \imcone$, $\mu^\star$ is optimal if and only if $A\mu^\star = \data$.

\section{Properties of minimisers}
\label{sec_min}
In this section, we gather results concerning the functional $\loss$ and its minimisers, proving that they are sparse when the data $\data$ is not in the image cone $\imcone$. 
First, we note that the functional $\loss$ defined by \eqref{eq:defloss} is a convex and proper function. 

\subsection{Characterisation of optimality}

We now derive necessary and sufficient \emph{optimality conditions}.

We first prove that any optimum must have a fixed unit mass. 

\begin{proposition}
\label{prop:massone}
If $\mu^\star$ is optimal for \eqref{eq:optprob}, then $\scl{\mu^\star}{1} =\int_K \dd \mu^\star = 1$. 
\end{proposition}

\begin{proof}
%
For any $\mu \in \MM_+$, we have
\begin{equation}
  \label{eq:massproof}
  \loss(\mu) = -\sum_{i=1}^{\ndet} \data_i \log\left(\frac{\scl{\mu}{a_i}}{\scl{\mu}{1}}\right) + \big(\scl{\mu}{1}- \log(\scl{\mu}{1})\big).
\end{equation}
Observe that the second term depends only on the mass $\scl{\mu}{1}$,
whereas the first term is scale-invariant.
As a result, an optimal $\mu$ has to minimise the second term,
which turns out to admit the unique minimiser $\scl{\mu}{1}= 1$.
\end{proof}

\begin{remark}
  This result follows from the optimality conditions derived later in \autoref{prop:KKT}, but the proof above is simple and also highlights that the maximum likelihood estimator for $\mu$ is consistent with the maximum likelihood estimator for $\int_K  \dd\mu$, as
  the second term in \eqref{eq:massproof} is none other than the negative log-likelihood of the total mass. 
\end{remark}

\begin{corollary}
The infimum of $\loss$ is a minimum.
\end{corollary}

\begin{proof}
From \autoref{prop:massone}, we may restrict the search of optimal solutions to $\setc*{\mu \in \MM_+}{\mu(K) = 1}$, which by the Banach--Alaoglu theorem, is weak-$\ast$ compact. Since $\loss$ is weak-$\ast$ continuous, the claim follows.
\end{proof}

We now give the full optimality conditions. The convex function $\loss$ defined in \eqref{eq:defloss} has the following open domain:
\[\operatorname{dom}(\loss) := \setc{\mu \in \MM_+}{ \scl{\mu}{a_i}> 0 \text{ for } i \in \suppdata}
  .
\]
Notice further that for any $\mu \in \operatorname{dom}(\loss)$, the function $\loss$ is Fr\'echet-differentiable (in the sense of the strong topology).
Its gradient is given
for $\mu\in\operatorname{dom}(\loss)$ is then
the element in the dual $\MM^\ast$ of $\MM$  given by
\begin{equation}
  \label{eq:gradloss}
  \nabla \loss(\mu) 
  = 1- \sum_{i =1}^\ndet \frac{\data_i a_i}{\scl{\mu}{a_i}}
  ,
\end{equation}
which we identify with an element of $\Cont(K)$.


 For any vector $w \in \RR^{\ndet}$, we define $\lambda(w) \in \RR^{\ndet}$ by
 \begin{equation}
   \label{eq:deflambda}
   \lambda_i(w) \coloneqq 1 - \frac{\data_i}{w_i},
 \end{equation}
 (with the convention $\lambda_i = 1$ if $\data_i = 0$, that is, if $i \notin \suppdata$).
  Using $\sum_{i=1}^{\ndet} a_i = 1$, we can rewrite \eqref{eq:gradloss} as 
  \begin{equation}
    \label{eq:mlemcont}
    \nabla \loss(\mu) = A^* \lambda(A \mu) =      \sum_{i=1}^\ndet \lambda_i(A \mu) \, a_i
 .
\end{equation}

\begin{proposition}
  \label{prop:KKT}
The measure $\mu^\star \in \MM$ is optimal for the problem \eqref{eq:optprob} if and only if the following optimality conditions hold 
\begin{equation}
  \label{KKT}
  \begin{aligned}
   A^* \lambda(A \mu^\star)  &\geq 0 \quad\text{on}\quad K, \\
   A^* \lambda(A \mu^\star)  &= 0 \quad\text{on}\quad \operatorname*{supp}(\mu^\star).
  \end{aligned}
\end{equation}
These conditions can be equivalently written as
\begin{equation}
  \begin{aligned}
    \label{KKT_}
    \sum_{i=1}^\ndet \frac{\data_i a_i}{\scl{\mu^\star}{a_i}} &\leq 1 \quad\text{on}\quad K, \\
       \sum_{i=1}^\ndet \frac{\data_i a_i}{\scl{\mu^\star}{a_i}} &= 1 \quad\text{on}\quad \operatorname*{supp}(\mu^\star).
  \end{aligned}
\end{equation}
\end{proposition}

\newcommand*\charfunc[1][\MM_+]{\chi_{#1}}

Recall that the \emph{normal cone} to $\MM_+$ at $\mu$ is defined as 
\[N_{\MM_+}(\mu) := \setc[\big]{f \in \Cont(K)}{\forall \nu \in \MM_+, \, \pairing{\nu - \mu}{f}\leq 0}
  .\] 
We need a characterisation of that normal cone before proceeding further.

\begin{lemma}
  \label{charact_normal_cone}
  The normal cone at a given $\mu \in \operatorname{dom}(\loss)$ is given by 
  \[N_{\MM_+}(\mu) = \setc*{f \in \Cont(K)}{ f \leq 0 \text{ on } K,\; \\
      f = 0 \text{ on } \operatorname*{supp}(\mu)}.\] 
\end{lemma}
\begin{proof} Let $f \in \Cont(K)$ be in $N_{\MM_+}(\mu)$, i.e., it satisfies $\pairing{\nu-\mu}{f} \leq 0$ for all $\nu$ in $\MM_+$.
  First, we choose $\nu = \mu + \delta_{x}$ (with $\delta_x$ the Dirac mass at $x$), which yields $f(x) \leq 0$, so we must have $f \leq 0$ on $K$.
  Then with $\nu = 0$, we find $\pairing{\mu}{f} \geq 0$.
  Since we also have $f \leq 0$, $\scl{\mu}{f} = 0$ leading to $f=0$ on $\operatorname{supp}(\mu)$.

  The reverse is also true: if $f \leq 0$ on $K$ and $f = 0$ on $\operatorname{supp}(\mu)$, then $\pairing{\nu-\mu}{f} \leq 0$ for all $\nu$ in $\MM_+$, which gives $f \in N_{\MM_+}(\mu)$. 
\end{proof}

\begin{proof}[Proof of \autoref{prop:KKT}]
%

Since $f$ is differentiable on $\operatorname{dom}(\loss)$ and convex on the convex set $\MM_+$, a point $\mu^\star \in \operatorname{dom}(\loss)$ is optimal if and only if 
\[ \nabla \loss(\mu^\star) \in - N_{\MM_+}(\mu^\star).\]

From the characterisation of $N_{\MM_+}(\mu)$ given below in~\autoref{charact_normal_cone}, and the fact that $\nabla \loss(\mu) = A^* \lambda(A \mu)$, the optimality condition exactly amounts to the conditions~\eqref{KKT}. 
\end{proof}

\begin{remark}
\label{dual_opt}
An alternative proof of these optimality conditions can be obtained by considering instead the equivalent problem of minimising $d(\data||w)$ with $w$ ranging over the cone $\imcone$. The cone has a non-empty relative interior which proves that Slater's condition is fulfilled. Since the problem is convex, KKT conditions are equivalent to optimality for $\min_{w \in \imcone} d(\data||w)$~\cite{Boyd2004}. 

The Lagrange dual is given by $g(\lambda):= \min d(\data||w) - \pairing{\lambda}{w}$ for $\lambda \in \imcone^\ast$.
A straightforward computation leads to
\begin{equation}
\label{dual_fct}
g(\lambda) = \sum_{i=1}^{\ndet} \data_i \log(1-\lambda_i),
\end{equation}
for $\lambda \leq 1$, with value $-\infty$ if there exists $i \in \operatorname{supp}(\data)$ such that $\lambda_i = 1$.

The KKT conditions for a primal optimal $w^\star$ and dual optimal $\lambda^\star$ write 
\begin{enumerate}[label=(\roman*)]
\item $w^\star \in \imcone$, $\lambda^\star \in \imcone^\ast$
\item\label{it:prodcond} $\scl{\lambda^\star}{w^\star} = 0$,
\item $\nabla_{w} d(\data|| w^\star) - \lambda^\star = 0$ (equivalent to \( \lambda^\star = \lambda(w^{\star}) = 1 -\frac{\data}{w^\star}\))
\end{enumerate}

A measure $\mu^\star$ is then optimal if and only if $A \mu^\star = w^\star$ for $w^\star$ primal optimal.
Since \(({\lambda^{\star}},{A \mu^{\star}}) = \pairing{\mu^{\star}}{A^* \lambda^{\star}}\) (by definition of $A^*$),
the condition~\ref{it:prodcond} thus becomes
\[
  \pairing{\mu^{\star}}{A^* \lambda^{\star}} =  \int_K A^* \lambda^\star \, d \mu^\star = 0
  .
\]
Since $\lambda^\star \in \imcone^\ast$, $A^* \lambda^\star \geq 0$ over $K$. Thus, we must have $A^* \lambda^\star = 0$ on $\operatorname{supp}(\mu^\star)$ for the above integral to vanish. All in all, we exactly recover the conditions~\eqref{KKT_}, with the additional interpretation that $\lambda(A \mu^\star)$ is a dual optimal variable.
\end{remark}

With these notations concerning the dual problem now set, let us prove that the dual problem has a unique maximiser $\lambda^\star$.
\begin{lemma}
\label{unique_dual}
The dual problem \[\max_{\lambda \in \imcone^\ast} g(\lambda)\] has a unique maximiser.
\end{lemma}
\begin{proof}
The idea is to go back the the primal problem by using the identity 
$\lambda^\star = 1 -\frac{y}{w^\star}$ for an optimal pair $(w^\star, \lambda^\star)$.
Since $w^\star$ relates to an optimal measure $\mu^\star$ by $A \mu^\star = w^\star$, we are done if we prove that $\setc{A \mu^\star}{\mu^\star \, \text{optimal}}$ is reduced to a singleton.
This fact is proved in~\cite{Mair1996}-[Theorem 4.1], and we here gather the main ideas for completeness.
For two optimal measures $\mu$ and $\nu$, we integrate the first KKT condition of~\eqref{KKT_} on the support of $\nu$ to uncover
  \[ \sum_{i=1}^\ndet y_i \frac{(A \nu)_i}{(A \mu)_i} \leq 1,\]
  and we may of course exchange the roles of $\mu$ and $\nu$ in this inequality.

  Suppose now that a vector $c \in \R^\ndet$ with $c_i = 0$ for $i \in \operatorname{supp}(y)$ satisfies both $\sum_{i=1}^\ndet y_i c_i \leq 1$ and $\sum_{i=1}^\ndet y_i (1/c_i) \leq 1$.
  From that, one obtains
  $\sum_{i=1}^{\ndet} y_i \frac{(c_i -1)^2}{c_i} \leq 0$,
  from which we conclude that $c_i = 1$ for all $i$.
  Applying this to $c = \frac{A \mu}{A \nu}$, the result is proved.
\end{proof}

\subsection{Case $\data \notin \imcone$}
When the data $\data$ is not in the cone $\imcone$, optimality conditions imply sparsity of any optimal measure.

\begin{corollary}

\label{cor_sparse}
Assume that $\data \notin \imcone $. Then any $\mu^\star$ minimiser of \eqref{eq:optprob} is sparse, in the following sense
\beq
\label{sparse}
\operatorname{supp}(\mu^\star) \subset  \argmin \bigg(\sum_{i=1}^{\ndet} \lambda_i^\star a_i\bigg),
\eeq
where $\lambda^\star$ is the unique maximiser for the dual problem, which satisfies $\lambda^\star \neq 0$. 
\end{corollary}

\begin{proof}
  Given an optimal $\mu^\star$, conditions \eqref{KKT} imply $\operatorname{supp}(\mu^\star) \subset  \argmin (A^\ast \lambda(A \mu^\star))$, with $\lambda(A \mu^\star) = 1- \frac{y}{A \mu^\star}$, where we used $\sum_{i=1}^{\ndet} a_i = 1$. The uniqueness of maximisers for the dual problem established in~\autoref{unique_dual}, allows us to write $\lambda(A \mu^\star) = \lambda^\star$. 



  The vector $\lambda^\star$ must be non-zero: if it weren't the case, then, using the definition \eqref{eq:deflambda},
  that would imply $\data \in \imcone$, which would contradict our initial assumption.

\end{proof}


\begin{remark}
  Why does condition \eqref{sparse} imply sparsity?
  Let $\lambda^\star$ be defined as in the previous theorem, and define the function $\varphi^\star \coloneqq A^* \lambda^\star= \sum_{i=1}^{\ndet} \lambda_i^\star \, a_i$.
 We know from \autoref{prop:KKT} that both $\varphi^\star \geq 0$ and $\operatorname{supp}(\mu^{\star}) \subset \argmin(\varphi^\star)$.
 
 Assuming that the $a_i$'s are linearly independent in $\Cont(K)$, $\varphi^\star$ cannot vanish identically since $\lambda^\star \neq 0$. 
  Supposing further that that for all~$i$,~$a_i \in\Cont^2(K)$, we have
    \[\operatorname{supp}(\mu^{\star}) \cap \operatorname{int}(K) \subset \mathcal{S} \coloneqq \setc{x\in K}{\nabla\varphi^\star(x) = 0}
    .
    \]
We make the final assumption that the Hessian of $\varphi^\star$ is invertible at the points $x \in \argmin (\varphi^\star)$, which is equivalent to its positive definiteness since these are minimum points of $\varphi^\star$. This implies that $\mathcal{S}$ consists of \emph{isolated points}.
 Consequently, the restriction to $\operatorname{int}(K)$ of any optimal solution $\mu^{\star}$ is a sum of Dirac masses.

 Note that all the above regularity assumptions hold for \emph{generic} functions $a_i$.
One case where all of them are readily satisfied is when the functions $a_i$ are analytic with $K$ connected.

 In fact, if we go further and assume that $\argmin(\varphi^\star)$ is reduced to a singleton $\bar x$, then the set of optimal measures is itself a singleton, given by the Dirac mass at $\bar x$.
 
\end{remark}

\begin{remark}
\label{super_sparse} 
We can exhibit a case where only Dirac masses are optimal.
Suppose that only $\data_{\iz} = 1$.
Then the function $\loss$ for measures $\mu$ such that $\mu(K) = 1$ is simply $\loss(\mu) = 1 - \log(\scl{\mu}{a_{\iz}})$.
One can directly check that a minimiser $\mu^\star$ necessarily satisfies $\operatorname{supp}(\mu^\star) \subset \argmax(a_{\iz})$, in agreement with condition~\eqref{sparse}.
If this set is discrete, then $\mu^\star$ is a sum of Dirac masses located at these points.
Note that such a data point $\data$ is outside  the cone $\imcone$ if and only if $\max(a_{\iz}) < 1$.
If not, it lies on the boundary of the cone, showing that some boundary points might lead to sparse minimisers as well.
\end{remark}

\subsection{Moment matching problem and case $\data \in   \operatorname{int}(\imcone)$}

When the data $\data$ is in the cone $\imcone$, searching for minimisers of \eqref{eq:optprob} is equivalent to solving $A \mu = \data$ for $\mu \in \MM_+$.
For the applications, we are particularly interested in the existence of absolutely continuous solutions. We make use of the results of~\cite{Georgiou2005}, which addresses this problem. 

We shall use the assumption:
\beq
\label{lin_ind}
\text{the functions $a_i$, $i=1, \ldots, \ndet$ are linearly independent in $\Cont(K)$}.
\eeq
Under~\eqref{lin_ind}, $\imcone$ has non-empty interior. 


We now recall a part of Theorem 3 of~\cite{Georgiou2005} which will be sufficient of our purpose.
\begin{theorem}[\cite{Georgiou2005}]
\label{exist_cont}
Assume that $\imcone$ and its dual cone $\imcone^\ast$ have non-empty interior.
Then for any $\data \in \operatorname{int}(\imcone)$, there exists $\mu^\star$ which is absolutely continuous, with positive and continuous density, such that $A \mu^\star = \data$.
\end{theorem}

\begin{lemma}
\label{smooth_int}
Under hypothesis \eqref{lin_ind}, $\imcone$ has non-empty interior, and if $\data \in \operatorname{int}(\imcone)$, there exists an optimal measure $\mu^\star$ which is absolutely continuous with positive and continuous density.
\end{lemma}

\begin{proof}
  This is a direct consequence of \autoref{exist_cont}. We just need to check that $\imcone^\ast$ has non-empty interior.
  Using the characterisation of the dual cone \eqref{eq:propdualcone}, this is straightforward since $\sum_{i=1}^{\ndet} a_i = 1$. 
\end{proof}

\subsection{Case $\data \in \partial \imcone$}
The previous approach settles the case where the data $\data$
is in the interior $\operatorname{int}(\imcone)$ of the image cone,
which poses the natural question of its boundary $\partial \imcone$.
It routinely happens in practice that some components of the data $\data$ are zero, which means that the vector $y$ lies at the border of the cone, \(\data \in \partial \imcone\).
Upon changing the compact, a further use of the results of~\cite{Georgiou2005} shows that if
the support of the data $\suppdata$ is not too small
(see the precise condition~\eqref{positive_2} below),
the situation is the same as for $\operatorname{int}(\imcone)$.

The idea is to remove all the zero components of the data vector $\data$, consider only the positive ones and try to solve $\scl{\mu^\star}{a_i} = \data_i$ for $i \in \suppdata$, while making sure that the measure $\mu^\star$ has a support such that $\scl{\mu^\star}{a_i} = 0$ for $i \notin \suppdata$.

We denote $\tilde{\ndet}:= \#(\suppdata)$, $\tilde{K} := K  \backslash \cup_{i \notin \suppdata} a_i^{-1}(\{0\})$, $\tilde{\data} = (\data_i)_{i \in \suppdata}$, and finally the reduced operator,
\begin{align}
\tilde{A} \colon   \mathcal{M}(\tilde{K}) &\longrightarrow  \mathbb{R}^{\tilde{\ndet}} \\
\mu  & \longmapsto \bigg(\int_K a_i \mu\bigg)_{i \in \suppdata},
\end{align}
which has an associated cone
$\tilde{A}(\mathcal{M}_+(\tilde{K}))$. 

We will need the assumptions 
\beq
\label{lin_ind_2}
\text{the functions $a_i$, $i \in \suppdata$ are linearly independent in $\mathcal{C}(\tilde{K})$}, 
\eeq
and 
\beq
\label{positive_2}
\sum_{i \in \suppdata} a_i > 0 \text{ on } \tilde{K}.
\eeq

\begin{proposition}
\label{smooth_bd}
We assume \eqref{lin_ind_2} and \eqref{positive_2}. $\tilde{A}(\mathcal{M}_+(\tilde{K}))$ has non-empty interior and we assume
\[\tilde{\data} \in \operatorname{int}\paren[\big]{\tilde{A}(\mathcal{M}_+(\tilde{K}))}.\] 
Then there exists an absolutely continuous solution  $\mu^\star$ of $A \mu = \data$ with positive and continuous density (on $\tilde K$). 
\end{proposition}

\begin{proof}
  We make use of \autoref{exist_cont}.
  In order to do so, we need the dual cone of $\tilde{A}(\mathcal{M}_+(\tilde{K}))$  to have a nonempty interior, which \eqref{positive_2} entails.
  Then we may build an absolutely continuous $\tilde{\mu}^\star \in  \mathcal{M}_+(\tilde{K})$ with positive and continuous density, such that $ \tilde{A} \tilde{\mu}^\star = \tilde{\data}$.
We then extend $\tilde{\mu}^\star$ to a measure on the whole of $K$ by defining $\mu^\star$ to equal $\tilde{\mu}^\star$ on $\tilde{K}$ with support contained in $\tilde{K}$, namely $\mu^\star(B) = \tilde{\mu}^\star(B \cup \tilde{K})$ for any Borel subset $B$ of $K$. Then $\mu^\star$ clearly solves $A \mu =  \data$ and thus minimises $\loss$. 
\end{proof}

Note that \autoref{smooth_int} is a particular case of \autoref{smooth_bd}, but we believe this presentation makes the role of $\operatorname{int}(\imcone)$ and $\partial \imcone$ clearer. 

Let us now finish this section by proving that not any point of the boundary may be associated to absolutely continuous measures. We denote $S$ the simplex in $\mathbb{R}^{\ndet}$, i.e.,
\beq
\label{simplex}
S:= \setc*{w \in \mathbb{R}^{\ndet}, w \geq 0}{\sum_{i=1}^{\ndet} w_i = 1}.
\eeq

\begin{proposition}
\label{extremal}
Assume that $\data \in \partial \imcone$ is an extremal point of $\imcone \cap S$. Then any measure satisfying $A \mu = \data$ is a Dirac mass.
\end{proposition}

We omit the proof, which is straightforward and relies on the linearity of the operator $A$ and the fact that the only extremal points among probability measures are the Dirac masses~\cite{Rudin1991}.

\section{Properties of ML-EM}
\label{sec_mlem}
We now turn our attention to the ML-EM algorithm \eqref{eq:mlemiterate} for
the minimisation of the functional $\loss$ (problem \eqref{eq:optprob}).

\subsection{Monotonicity and asymptotics}
We first proceed to prove that the algorithm is monotonous,
a property stemming from it being an expectation-maximisation algorithm.  

We build a so-called \emph{surrogate} function, i.e., a function $Q_k$ such that $\loss(\mu) \leq Q_k(\mu)$ for all $\mu$, with equality for $\mu = \mu_k$, where $Q_k$ is minimised at $\mu_{k+1}$.
The precise details are in \autoref{prop:surrogate}.

\begin{lemma}
  \label{prop:surrogate}
For a given $k \in \NN$, we define
\[
  X_k:= \setc[\Big]{\mu \in \MM_+}{\mu_{k+1} \ll \mu \ll \mu_k,\quad \pairing{\mu}{1} = 1}
.
\]
For a measure $\mu\in X_k$, and for $i = 1,\dotsc, \ndet$, we define the probability distribution
\[\nu_i(\mu) \coloneqq \frac{a_i \mu}{\pairing{\mu}{a_i}}
  .
  \]
  as well as
\[
  Q_k(\mu) \coloneqq \loss(\mu) + \sum_{i=1}^\ndet \data_i D\paren[\big]{\nu_i(\mu_k) || \nu_i(\mu)}
  .
\]


The following holds:
\begin{enumerate}[label=\upshape(\roman*)]
  \item\label{surr:it:major} \(Q_k(\mu) \geq \loss(\mu), \qquad \mu \in X_k\)
\item\label{surr:it:contact} $Q_k (\mu_k) = \loss(\mu_k)$
\item\label{surr:it:div} \(Q_k(\mu) - Q_k(\mu_{k+1}) = D(\mu_{k+1}|| \mu), \qquad \mu \in X_k
  \)
\end{enumerate}
\end{lemma}

\begin{proof}
  The fact that $D(\nu_i(\mu_k) || \nu_i(\mu))$ for $i = 1,\dotsc,\ndet$ are divergences
  allows us to conclude about \ref{surr:it:major} and \ref{surr:it:contact}.

  After defining
\begin{equation}
  \data_i^k \coloneqq  \scl{\mu_k}{a_i}
  ,
  \qquad
  1 \leq i \leq \ndet
\end{equation}
and using the definition of $\loss$ in \autoref{eq:defloss},
we compute
\[
  Q_k(\mu) = 1 - \sum_{i=1}^\ndet \data_i \pairing[\Big]{\nu_i(\mu_k)}{\log\paren[\Big]{y_i^k \frac{\dd \mu}{\dd \mu_k}}}
  ,
  \qquad
  \mu \in X_k
\]

This gives
\begin{align*}
  Q_k(\mu) - Q_k(\mu_{k+1}) & =  \sum_{i=1}^{\ndet}  \data_i \pairing[\bigg]{\nu_i\paren{\mu_k}}{\log\paren[\Big]{\data_i^k \frac{\dd\mu_{k+1}}{\dd\mu_k}} - \log\paren[\Big]{\data_i^k \frac{\dd\mu}{\dd\mu_{k}}}}     \\
                            & =  \sum_{i=1}^{\ndet}  \data_i \pairing[\bigg]{\nu_i\paren{\mu_k}}{\log\paren[\Big]{ \frac{\dd\mu_{k+1}}{\dd\mu}} }     \\
                            & =   \pairing[\bigg]{\underbrace{\sum_{i=1}^{\ndet}  \data_i\nu_i\paren{\mu_k}}_{\mu_{k+1}}}{\log\paren[\Big]{ \frac{\dd\mu_{k+1}}{\dd\mu}} }     \\
                                & = D\paren[\big]{\mu_{k+1}|| \mu}.
\end{align*}
which proves \ref{surr:it:div}.
\end{proof}

\begin{corollary}
\label{cor_dist}
For any $\mu_0 \in \operatorname{dom}(\loss)$, we have
\[D\big(\mu_{k+1} ||\mu_k\big) \leq \loss(\mu_k) - \loss(\mu_{k+1}).\]
In particular,
\[
\loss(\mu_{k+1}) \leq \loss(\mu_k)
\]
\end{corollary}

\begin{proof}
  First, observe that $\mu_{k+1} \in X_k$.
  Now, from \ref{surr:it:contact} and \ref{surr:it:major} in \autoref{prop:surrogate}, we obtain $Q_k(\mu_k) - Q_k(\mu_{k+1}) = \loss(\mu_k) - Q_k(\mu_{k+1}) \leq \loss(\mu_k) - \loss(\mu_{k+1})$.
  We conclude using \ref{surr:it:div}.
  
\end{proof}


Let us now prove that all cluster points of ML-EM are fixed points of the algorithm.
\begin{proposition}
\label{Fixed_point}
For any $\mu_0 \in \operatorname{dom}(\loss)$, any cluster point $\bar \mu$ of ML-EM is a fixed point of the algorithm, namely
\[\bar \mu =  \bar \mu \left(\sum_{i=1}^{\ndet} \frac{y_i a_i}{\scl{\bar \mu}{a_i}}\right)
.
\]
\end{proposition}

\begin{proof}

  We pursue an idea from~\cite{Mair1996}.
 Since we have $\mu_K(K) = \scl{\mu_k}{1} = \sum_{i=1}^{\ndet} \data_i = 1$ for all $k \geq 1$, $(\mu_k)$ is a bounded sequence in $\MM_+$.
  By the Banach--Alaoglu Theorem, it is thus weak-$\ast$ compact in $\MM_+$, and we may extract some subsequence $\mu_{\varphi(k)}$ converging to a weak-$\ast$ cluster point $\bar{\mu}$ of ML-EM. 
  Note first that $\loss$ is weak-$\ast$ continuous.

  We now observe that such a cluster point must satisfy $\bar \mu \in \operatorname{dom}(\loss)$.
  Indeed, $\scl{\bar\mu}{a_i} > 0$ for any $i \in \suppdata$ ({i.e.}, whenever $\data_i>0$).
  Otherwise, $\loss$ would go to infinity, a contradiction with the fact that $\loss$ decreases along iterates and $\loss(\mu_0) < +\infty$. 
  
We also note that the convergence of $\loss(\mu_k)$ towards $\loss(\bar \mu)$ is then along the whole sequence since $\setc{\loss(\mu_k)}{k=0,\dotsc}$ is decreasing.

Upon extracting another subsequence, we may assume that the the sequence $(\mu_{\varphi(k) + 1})$ is also convergent, say to $\tilde \mu \in \operatorname{dom}(\loss)$.
Passing to the limit in the defining relation of ML-EM (as one readily checks that it is weak-$\ast$ continuous) along the subsequence, we find 
\[\tilde \mu =  \bar \mu \left(\sum_{i=1}^{\ndet} \frac{y_i a_i}{\scl{\bar \mu}{a_i}}\right),\] and all it remains to show is that $\tilde \mu = \bar \mu$.

The inequality established in \autoref{cor_dist} becomes
\[D\paren[\big]{\mu_{\varphi(k)+1} || \mu_{\varphi(k)}} \leq \loss(\mu_{\varphi(k)}) - \loss(\mu_{\varphi(k)+1}).\]
The right-hand side converges to $0$. For the left-hand side, we use the property that the function $(\mu, \nu) \mapsto D(\mu|| \nu)$ is weak-$\ast$ lower semi-continuous~\cite{Posner1975}. This leads to $D(\tilde \mu| \bar \mu) \leq 0$, whence $\tilde \mu = \bar \mu$.
\end{proof}

Note that
\[\bar \mu =  \bar \mu \left(\sum_{i=1}^{\ndet} \frac{y_i a_i}{\scl{\bar \mu}{a_i}}\right) \iff \sum_{i = 1}^{\ndet} \frac{\data_i a_i}{\scl{\bar\mu}{a_i}} = 1 \; \; \text{on} \;\; \operatorname*{supp}(\bar \mu).\]

Thus, ML-EM cluster points satisfy one of the two optimality conditions \eqref{KKT_}.
Although we conjecture they actually satisfy both of them under the additional hypothesis that $\operatorname{supp}(\mu_0)  = K$, we are able to prove it only when $y \in C$.

\subsection{Case $\data \notin \imcone$}

When $\data \notin \imcone$, we know from \autoref{cor_sparse} that optimal solutions are sparse.
Note that this is also the case for boundary points which are extremal in $\imcone \cap S$, in virtue of \autoref{extremal}. 

We do not know whether ML-EM iterates converge to an optimal point, but we can at least state a straightforward partial sparsity result from the first optimality condition, which we call \textit{weak sparsity}.
\begin{corollary}
\label{weak_sparsity}
Assume that $\data \notin \imcone$ and the linear independence condition~\eqref{lin_ind}. Then, for $\mu_0 \in \operatorname{dom}(\loss)$, any cluster point $\bar \mu$ of ML-EM is such that 
\[ \operatorname{supp}(\bar \mu) \subset \left(\sum_{i=1}^{\ndet} \lambda_i(A \bar \mu) \, a_i\right)^{-1}(\{0\}),\]
with $\lambda(A \bar \mu) \neq 0$ and the components $\lambda_i(A \bar \mu)$'s do not have the same sign. 
In particular, $\operatorname{supp}(\bar \mu) \neq K$.
\end{corollary}

\begin{proof}
This is just a rephrasing of~\autoref{Fixed_point}, using the formula for $\lambda(\mu)$ given in \autoref{eq:deflambda}.
\end{proof}

\begin{remark}
In general, the fact that the components $\lambda_i(A \bar \mu)$'s do not have the same sign will impose that $\operatorname{supp}(\bar \mu)$ is restricted to a lower dimensional set, of Lebesgue measure $0$. Thus, one cannot expect that the cluster points of ML-EM are absolutely continuous when $y \notin \imcone$.
\end{remark}

We can go further in the case where $\data_i = 0$ except for $\data_{\iz} = 1$, for which we saw in \autoref{super_sparse} that any minimiser $\mu^\star$ of the function $\loss$ for normalised measures, i.e., $\loss(\mu) = 1 -\log(\scl{\bar \mu}{a_{\iz}})$, will be such that $\text{supp}(\mu^\star)  \subset \argmax a_{\iz}$. The goal of this subsection is to highlight a case where one clearly identifies the limiting measure of ML-EM, and its dependence with respect to the initial measure~$\mu_0$. This suggests that in the sparse case $\data \notin \imcone$, the position of Dirac masses will in general depend on the initial condition~$\mu_0$.

We recall Laplace's method (see~\cite{Wong2001}) which holds for $f \in \Cont(K)$, $g \in \Cont^2(K)$ with a single non-degenerate interior maximum point $\bar{x}$, and reads:
\beq
\label{Laplace}
\int_K f(x) e^{g(x) t} \, \dd x \sim (2 \pi)^{p/2} \frac{f(\bar{x})}{\sqrt{|\det(H(\bar{x}))|}} \frac{e^{g(\bar{x}) t}}{t^{\frac{p}{2}}}\quad \text{as } t \to \infty, 
\eeq
where $H(\bar{x})$ is the Hessian of $g$ at $\bar{x}$.

\begin{proposition}
\label{Dirac}
Assume that $\data_i = 0$ for $i \neq i_0$, $\data_{\iz} = 1$.
Assume further that $\argmax a_{\iz} = \set{\bar{x}_1, \dotsc, \bar{x}_l}$ with $\bar{x}_j \in \operatorname{int}(K)$ for all $j = 1, \ldots, l$,
that $a_{\iz}$ is of class~$\Cont^2$
and that the maximum points $\bar{x}_j$ are non-degenerate.
Under these assumptions and for $\mu_0$ absolutely continuous with continuous positive density (still denoted $\mu_0$), the ML-EM sequence $(\mu_k)$ satisfies
\[
  \mu_k \rightharpoonup \mu^\star :=C \sum_{j=1}^l \frac{\mu_0(\bar{x}_j)}{\sqrt{|\det H_j}|} \delta_{\bar{x}_j}
  ,
\]
where $C>0$ is a normalising constant such that the limit has mass one, $H_j$ is the Hessian of $a_{\iz}$ at the point $\bar{x}_j$, and $\delta_{\bar{x}_j}$ is the Dirac mass centred at $\bar{x}_j$.
\end{proposition}

\begin{proof}
We first remark that the ML-EM iterates are then explicitly solved as 
\[ \mu_k =  \frac{a_{\iz}^k \mu_0}{\int_K a_{\iz}^k(x) \mu_0(x) \, \dd x}. \]

We denote $M:= \max_{x \in K} \log (a_{\iz}(x))$ and let $f \in \Cont(K)$ be a generic function. For $\delta$ small enough such that, for all $j$, $\bar{x}_j$ is the unique maximum point of $a_{\iz}$ in $B(\bar{x}_j, \eta)$, we split contributions in the integral of $a_{\iz}^k \mu_0$ against $f$ as follows 
\begin{align*}
\langle a_{\iz}^k \mu_0 ,f \rangle & = \int_K f(x) \mu_0(x) e^{k \log(a_{\iz}(x))} \,\dd x  \\
& = \sum_{j=1}^l \int_{B(\bar{x}_j, \eta)} f(x) \mu_0(x) e^{k \log(a_{\iz}(x))} \,\dd x \\
& \qquad \qquad \qquad \qquad \qquad  + \int_{K \backslash \cup_{j=1}^l B(\bar{x}_j, \eta)} f(x) \mu_0(x) e^{k \log(a_{\iz}(x))} \,\dd x.
\end{align*}
From Laplace's method~\eqref{Laplace}, each term in the first sum can be estimated as
\[ \int_{B(\bar{x}_j, \eta)} f(x) \mu_0(x) e^{k \log(a_{\iz}(x))} \,\dd x  \sim (2 \pi)^{p/2} \frac{f(\bar{x}_j)\mu_0(\bar{x}_j)}{\sqrt{|\det H_j}|} \frac{e^{M k}}{k^{\frac{p}{2}}} \quad \text{as } k \to \infty,\]
whereas  one can check that the second term is $o\big( e^{Mk}\big)$.
We end up with 
\[
\langle a_{\iz}^k \mu_0 ,f \rangle \sim (2 \pi)^{p/2} \sum_{j=1}^l \left( \frac{f(\bar{x}_j)\mu_0(\bar{x}_j)}{\sqrt{|\det H_j}|}\right) \frac{e^{M k}}{k^{\frac{p}{2}}} \quad \text{as } k \to \infty.
\]

Applying this equivalent for $f=1$ yields an equivalent for the denominator $\int_K a_{\iz}^k(x) \mu_0(x) \, \dd x$ in the explicit formula for $\mu_k$, which is of the order of $e^{M k}/k^{\frac{p}{2}}$. All in all, we find 
\[
\langle \mu_k ,f \rangle \longrightarrow C \sum_{j=1}^l \left( \frac{f(\bar{x}_j)\mu_0(\bar{x}_j)}{\sqrt{|\det H_j}|}\right)  = \left\langle C \sum_{j=1}^l \frac{\mu_0(\bar{x}_j)}{\sqrt{|\det H_j}|} \delta_{\bar{x}_j}, f \right\rangle = \langle \mu^\star ,f \rangle,
\]
as $k \to \infty$, with $C>0$ normalising $\mu^\star$, which proves the claim.

\end{proof}

\subsection{Case $\data \in \imcone$}

When the data $\data$
is in the cone $\imcone$, there are infinitely many measures satisfying $A \mu = \data$, and when there are absolutely continuous ones, a desirable property of ML-EM is to converge to one of them rather than to a measure having a singular part.
In order to address this question, we start with a Proposition of independent interest, valid for any data $\data$.
It generalises a result which holds in the discrete case. 
It gives information on the divergence of ML-EM iterates to \textit{any} fixed point of the ML-EM algorithm. 

\begin{proposition}
\label{decrease} 
Let $\bar \mu$ be a fixed point of the ML-EM algorithm, and $\mu_0 \in \operatorname{dom}(\loss)$. We further assume that \[D(\bar \mu || \mu_0) < \infty.\]
Then the ML-EM iterates are such that
\[ \forall k \in \mathbb{N}, \; D(\bar \mu || \mu_{k+1}) \leq D(\bar \mu || \mu_k) - \loss(\mu_k) + \loss(\bar \mu) .\]
In particular, the KL divergence to any optimum decreases: 
\[ \forall k \in \mathbb{N}, \; D(\bar \mu || \mu_{k+1}) \leq D(\bar \mu || \mu_k) .
\]
\end{proposition}

\begin{proof}
  We recursively prove that $D(\bar \mu || \mu_k) < \infty$.
  Assuming this holds at the step \(k\), one checks that the definition of ML-EM iterates is such that if $\mu_k$ is absolutely continuous with respect to $\bar \mu$, then so is $\mu_{k+1}$.
Furthermore, 
\begin{align*}- \log\bigg(\frac{\dd \mu_{k+1}}{\dd \bar \mu}\bigg) & = - \log\bigg(\frac{\dd \mu_k}{\dd \bar \mu}\bigg) - \log\Bigg(\sum_{i=1}^{\ndet} \frac{\data_i a_i}{\scl{\mu_k}{a_i}} \Bigg) \\
& = - \log\bigg(\frac{\dd \mu_k}{\dd \bar \mu}\bigg) - \log\Bigg(\sum_{i=1}^{\ndet} \frac{\data_i a_i}{\scl{\mu_k}{a_i}} \Bigg) \sum_{i=1}^{\ndet} \frac{\data_i a_i}{\scl{\bar \mu}{a_i}} \\
& = - \log\bigg(\frac{\dd \mu_k}{\dd \bar \mu}\bigg) + \sum_{i=1}^{\ndet} \frac{\data_i a_i}{\scl{\bar \mu}{a_i}} \log\Bigg(\sum_{i=1}^{\ndet} \frac{\data_i a_i}{\scl{\bar \mu}{a_i}} \Bigg /\sum_{i=1}^{\ndet} \frac{\data_i a_i}{\scl{\mu_k}{a_i}}\Bigg), \numberthis \label{eq_approach}
\end{align*}
where we twice took advantage of $\sum_{i=1}^{\ndet} \frac{\data_i a_i}{\scl{\bar \mu}{a_i}} = 1$ on $\text{supp}(\bar \mu)$ by definition of a fixed point, a property we may use since we will eventually integrate against $\dd \bar \mu$.

Now, we use the convexity of $(u,v) \mapsto u \log(\frac{u}{v})$ on $[0,+\infty) \times (0,+\infty)$ (following an idea of~\cite{Iusem1992}) to bound the second term as follows
\begin{align*} 
\sum_{i=1}^{\ndet} \frac{\data_i a_i}{\scl{\bar \mu}{a_i}} \log\Bigg(\sum_{i=1}^{\ndet} \frac{\data_i a_i}{\scl{\bar \mu}{a_i}} \Bigg /\sum_{i=1}^{\ndet} \frac{\data_i a_i}{\scl{\mu_k}{a_i}}\Bigg) & \leq \sum_{i=1}^{\ndet} \data_i \, \frac{a_i}{\scl{\bar \mu}{a_i}} \log\Bigg(\frac{\scl{\mu_k}{a_i}}{\scl{\bar \mu}{a_i}}\Bigg).
\end{align*}
When integrated against $d\bar \mu$, the right hand side simplifies to 
\[\sum_{i=1}^{\ndet} \data_i  \log\Bigg(\frac{\scl{\mu_k}{a_i}}{\scl{\bar \mu}{a_i}}\Bigg) = - \loss(\mu_k) + \loss(\bar \mu). \]

Wrapping up, the integration of~\eqref{eq_approach} against $\dd \bar \mu$ and the above inequality exactly yield the result.

\end{proof}

\begin{remark}
As we saw, we typically expect a fixed point $\bar \mu$ to be a sparse measure when $\data \notin \imcone$, which means that the assumption that $\mu_0$ is absolutely continuous with respect to $\bar \mu$ will typically not be satisfied for $\mu_0$ chosen to be constant over $K$. This result will instead come in handy when $\data \in \imcone$. 
\end{remark}

\begin{corollary}
  \label{prop:mlemabscont}
Assume that $\data \in \imcone$, and that there exists an absolutely continuous measure $\mu^\star$ with positive and continuous density (on $\tilde K$) such that $A \mu^\star = \data$. Then, for any $\mu_0$ absolutely continuous, with a positive and continuous density, any cluster point $\bar \mu$ satisfies $\operatorname{supp}(\bar \mu) = \tilde K$. 
\end{corollary}

\begin{proof}
Let $\mu^\star$ be such an absolutely continuous optimum with positive and continuous density on $\tilde K$. The assumption on the initial measure $\mu_0$ ensures that it satisfies conditions \eqref{init_1}, i.e., $\mu_0 \in \operatorname{dom}(\loss)$, and since $\mu_1$ is then continuous and positive on~$\tilde K$, we also have $D(\mu^\star || \mu_1) < \infty$.
We may then use \autoref{decrease} to obtain
\[
  \forall k \in \mathbb{N}^\star, \; D(\mu^\star || \mu_{k+1}) \leq D(\mu^\star || \mu_k)
  .
\]
Let $\bar{\mu}$ be a cluster point of the iterates $\{\mu_k\}$. Note that we obviously have $\operatorname{supp}(\bar \mu) \subset \tilde K$.
By weak-$\ast$ lower semi-continuity, we may pass to the limit in the inequality to obtain 
\[  D(\mu^\star || \bar{\mu}) \leq \lim_{k \to +\infty} D(\mu^\star || \mu_k).\]

In particular, $D(\mu^\star || \bar{\mu}) < +\infty$, which by definition implies that $\mu^\star \ll \bar{\mu}$, and consequently $\bar \mu$ has support at least $\operatorname{supp}(\mu^\star) = \tilde K$.  
\end{proof}
We are now in a position to prove the main result of this section, where we use the notations of \autoref{smooth_bd}.

\begin{theorem}
\label{thm_smooth}
Assume that conditions \eqref{lin_ind_2} and \eqref{positive_2} hold, and that $\tilde{\data} \in \operatorname{int}(\tilde{A}(\mathcal{M}_+(\tilde{K}))$.
Then, if the initial measure $\mu_0$ is absolutely continuous, with a positive and continuous density, the cluster points of the ML-EM iterates
\begin{enumerate}[label=\upshape(\roman*)]
  \item\label{thm_supp} have support $\tilde K$,
\item\label{thm_opt} are optimal.
  \end{enumerate} 
 In particular, the algorithm is convergent in the sense that \[l(\mu_k)\xrightarrow[k\rightarrow +\infty]{} \inf_{\mu \in \MM_+} l(\mu).\]

\end{theorem}

\begin{proof}
  Under these assumptions,  by \autoref{smooth_bd}, there exists an absolutely continuous measure $\mu^\star$ with positive and continuous density
on $\tilde K$, such that $A \mu^\star = \data$. We may thus apply~\autoref{prop:mlemabscont} to conclude that any cluster point must be absolutely continuous with support equal to $\tilde K$, showing~\ref{thm_supp}.
  
  Letting $\bar \mu$ be a cluster point, it remains to show the optimality~\ref{thm_opt}, namely that $A \bar \mu = \data$.   
  As a cluster point, it must be fixed point of the algorithm from~\autoref{Fixed_point}, i.e. $\bar \mu = \bar \mu \left(\sum_{i=1}^{\ndet} \frac{y_i a_i}{\scl{\bar \mu}{a_i}}\right)$. Since $\operatorname{supp}(\bar \mu) = \tilde K$, we obtain  \[ \sum_{i=1}^{\ndet} \left(1-\frac{y_i}{\scl{\bar \mu}{a_i}} \right) a_i  = 0,\]
  on $\tilde K$.
   By the linear independence assumption~\eqref{lin_ind_2}, this imposes $\scl{\bar \mu}{a_i} = y_i$ for all $i \in \{1, \ldots, \ndet\}$, whence the optimality of $\bar \mu$.
   
   As for the convergence of the algorithm, we recall that the whole sequence $(l(\mu_k))$ is decreasing and hence, converges. Extracting such that a subsequence of the ML-EM iterates converges to an optimal measure, the limit of $(l(\mu_k))$ is identified to be $ \inf_{\mu \in \MM_+} l(\mu)$.


\end{proof}

Note that these results also cover the case of \autoref{smooth_int}: if $\data \in \operatorname{int}(\imcone)$ and the functions $\set{a_i}$ are linearly independent, cluster points of ML-EM are absolutely continuous, whenever the initial measure $\mu_0$ is absolutely continuous with a positive and continuous density. 

\begin{remark}
This result does not mean that cluster points do not have a singular part, such as Dirac masses. However, the fact that  cluster points have full support is relevant in practice: cropping the image obtained by ML-EM algorithm allows one to reduce the effect of the singular part and to uncover the continuous one (assuming that the measure does not have any continuous singular part). This is also what should be obtained by smoothing the final image, or by regularising the functional~$\loss$, see~\cite{OSullivan1995} for an analysis of such regularisation techniques.
\end{remark}

%

\section{Statistics} 
\label{sec_stat}

In this section, we estimate the probability that the data $\data$ stays in the image cone, i.e., $\data \in \imcone$.

Let us first go back to modelling (before any normalisation) by introducing a \emph{dose variable} $t$.
We assume that the real image is given by $\mu_r \in \MM_+$ which represents the image in the relevant unit depending on the context ($Bq$ for PET).
The dosage $t$ gives rise to independent random variables $N_i \sim \mathcal{P}(\gamma_i t)$ where $\gamma_i := \scl{\mu_r}{a_i}$, whose sum $N = \sum_{i=1}^{\ndet} N_i$ is $\mathcal{P}(\gamma t)$ with $\gamma := \sum_{i=1}^{\ndet} \gamma_i$.

The expected frequencies are given by $\data_r := (\frac{\gamma_1}{\gamma}, \ldots, \frac{\gamma_{\ndet}}{\gamma})$, which is an element of $\imcone \cap S$, where we recall that $S$, given by~\eqref{simplex}, is the simplex in $\mathbb{R}^{\ndet}$.

With our previous notations, $n_i$ and $n$ are thus realisations of the random variables $N_i$, and $N$, and now $\data =\paren[\big]{\frac{N_1}{N}, \ldots, \frac{N_{\ndet}}{N}}$ is a random variable. We emphasise that it is an estimator for $\data_r$ which depends on $t$ by using the notation $\hat{\data}_t$. When conditioned on the fact that $N = n$, we will denote $\hat{\data}_n:= (\frac{N_1}{n}, \ldots, \frac{N_{\ndet}}{n})$.

By the law of large numbers, $\hat{\data}_t$ tends to $\data_r$ almost surely as $t \to +\infty$, so we know that if $\data_r \in \operatorname{int}(\imcone)$, a high-enough dose will ensure that $\hat{\data}_t \in \operatorname{int}(\imcone)$ as well, avoiding having only sparse measures as solutions to the maximum likelihood problem.

 We now wish to give quantitative bounds for $\mathbb{P}(\hat{\data}_t \notin \imcone)$, one with a conditioning on the number of events $n$, the other without such a conditioning. The aim is to address the following questions:
\begin{itemize}
\item a posteriori, for a given number of points $n$, how small is the probability that $\hat{\data}_n \notin \imcone$?
\item a priori, how large should the dosage $t$ be for the probability that $\hat{\data}_t \notin \imcone$ to be small enough?
\end{itemize}

The celebrated Sanov's Theorem~\cite{Sanov1961} states that the empirical distribution has an exponentially small probability of being in a set which does not contain the real distribution, where the exponential is controlled by the Kullback--Leibler divergence from the real distribution to the set.
 We thus define
\[\varepsilon := \inf_{q \in S \cap \imcone^c}  d(q || \data_r),\]
which is the Kullback--Leibler divergence of $\data_r$ to the boundary of the set $\imcone$ (intersected with the simplex $S$).

In both cases, we shall give two different bounds which might be relevant in different regimes in the parameters $(\ndet,n)$ and $(\ndet,t)$, respectively. 

\begin{proposition}
The following concentration bounds hold:
  \label{prop:nbound}
  \begin{equation}
    \mathbb{P}(\hat{\data}_n \notin \imcone) \leq \left\{
      \begin{array}{ll}
        (n+1)^{\ndet} \, e^{-n \varepsilon}, \\
        2 \ndet  \, e^{- n \varepsilon/\ndet}.
      \end{array}
    \right.
  \end{equation}
\end{proposition}
\begin{proof}
Conditioned on $N=n$, the random vector $\hat{\data}_n$ follows the multinomial distribution of parameters $n$ and $\data_r$. The first inequality is then nothing but a direct application of Sanov's Theorem~\cite{Sanov1961}. The second is more recent and given in Lemma~6 of~\cite{Mardia2018}.
\end{proof}

We now proceed to the case with dose $t$:

\begin{theorem}
\label{exp_dec}
The following concentration bounds hold:
\[
\mathbb{P}(\hat{\data}_t \notin \imcone) \leq \left\{
    \begin{array}{ll}
        C(\ndet) (1 + (\gamma t)^{\ndet}) \, e^{-\gamma t \varepsilon}, \\
        2 \ndet  \, e^{- \gamma t  \varepsilon/\ndet},
    \end{array}
\right.
\]
where $C(\ndet)$ is a combinatorial constant which depends only on $\ndet$ and satisfies 
$C(\ndet) \leq \big(\frac{a (\ndet+1)}{\log(\ndet+2)}\big)^{\ndet+1}$, with $a=0.792$.
\end{theorem}
\begin{proof}
We may write 

\[\mathbb{P}(\hat{\data}_t \notin \imcone)  =  \mathbb{E}(\mathbb{P}(\hat{\data}_N \notin \imcone) | N) \leq \mathbb{E}(g(N))\] where $g(n) = (n+1)^{\ndet} \, e^{-n \varepsilon}$ or $2 \ndet  \, e^{- n \varepsilon/\ndet}$ from the previous proposition.
It is now a matter of estimating this expectation with $N ~\sim \mathcal{P}(\gamma t)$. In the second case,  
\begin{align*}
\mathbb{E}(g(N)) & = 2\ndet \, e^{- \gamma t} \sum_{n=0}^{+\infty} e^{-n \varepsilon/\ndet}  \frac{(\gamma t)^n}{n!}  \\
& = 2\ndet \, e^{- \gamma t(1- \exp(-\varepsilon/\ndet))} \leq 2 \ndet \, e^{-\gamma t \varepsilon/\ndet},
\end{align*}
from $1- e^{-u} \geq u$.

In the first case, $\mathbb{E}(g(N)) = e^{- \gamma t} \varphi_{\ndet}(\gamma t e^{-\varepsilon}),$
with \[ \varphi_{\ndet}(x):=\sum_{n=0}^{+\infty} (n+1)^{\ndet} \, \frac{x^n}{n!},\] and $x:= \gamma t e^{-\varepsilon}$, which we now estimate.
We may integrate to find \[\int_0^x \varphi_{\ndet}(u) \, \dd u = \sum_{n=1}^{+\infty} n^{\ndet} \frac{x^n}{n!} =: T_{\ndet}(x) e^{x},\] where $T_{\ndet}$ is the so-called Touchard polynomial of order $\ndet$, which has degree $\ndet$. 

This allows to go back to $\varphi_{\ndet}(x)$ as $\varphi_{\ndet}(x) = \left(T_{\ndet}'(x) + T_{\ndet}(x)\right) e^{x} = \frac{T_{\ndet+1}(x)}{x} e^{x}$, using a well-known property of Touchard polynomials. The Touchard polynomial of order $\ndet$ has integer coefficients (the Stirling numbers), whose sum is given by the so-called Bell number~$B_{\ndet}$. 

Using the crude bound $P(x) = \sum_{k=0}^{\ndet} a_k x^k \leq \big(\sum_{k=0}^{\ndet} a_k\big) (1+x^{\ndet})$ valid for all $x \geq 0$ when $P$ is a polynomial with non-negative coefficients, we may write $\frac{T_{\ndet+1}(x)}{x} \leq B_{\ndet+1}(1+x^{\ndet})$ for $x \geq 0$.

Summing up, we have \begin{align*}
\mathbb{E}(g(N)) & \leq  B_{\ndet+1}e^{- \gamma t}(1+x^{\ndet}) e^x  = B_{\ndet+1} (1+(\gamma t e^{-\varepsilon})^{\ndet}) e^{- \gamma t(1- \exp(-\varepsilon))} \\
& \leq C(\ndet) (1 + (\gamma t)^{\ndet}) \, e^{-\gamma t \varepsilon},
\end{align*}
where $C(\ndet) := B_{\ndet+1}$.
The bound about Bell numbers such as $C(\ndet)$, stated in the proposition, can be found in~\cite{Berend2010}.
 We use them here as bounds on the moments of a Poisson random variable. 
\end{proof}

\begin{remark}
Although the bound coming from Sanov's theorem is sharper at the limit $n \to +\infty$ or $t \to +\infty$, it may be that the other one is relevant for realistic values $\ndet$, $n$ or $t$. If these bounds are taken as functions of $\varepsilon$, it can be checked that the alternative bound becomes more stringent in the regime where $\varepsilon \ll \frac{\ndet}{n} \log(n)$.
\end{remark}

\section{Numerical simulations}
\label{sec:numerics}

We perform simulations using the Python library Operator Discretization Library~\cite{Adler2017}.
The interested reader may run our numerical experiments using a Jupyter Notebook~\cite{notebook}.

All simulations are run with a 2D PET operator $A$ having \num{90} views and \num{128} tangential positions, leading to a number of (pairs of) detectors $\ndet = 11520$.
The image resolution is $256 \times 256$. 
We draw $\bar{\data}_t  \sim \frac{1}{t} \mathcal{P}(t A \mu_r)$ for different doses $t$, so that the higher $t$, the lower the noise level.
We then normalise $(y_t)_i \coloneqq (\bar{y}_t)_i/ \sum_i (\bar{y}_t)_i$ for $i =1, \ldots, \ndet$.

\begin{figure}
  \centering
  \includegraphics[width=\textwidth]{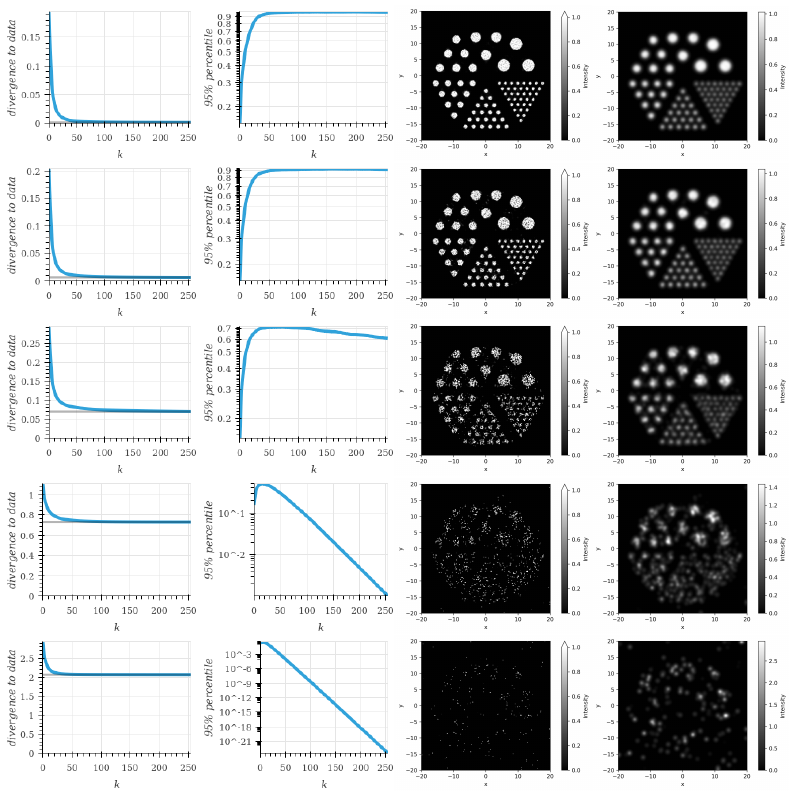}
  \caption{
    We show here various reconstructions for a decreasing amount of dose (the first row has $t = 10^{2}$ and each subsequent row has ten times less dose than the previous one).
    The columns depict
    (a) the divergence to the data $d(y_t || A \mu^k)$
    (b) the \SI{95}{\percent} percentile (logarithmic scale)
    (c) the reconstruction with limitations between zero and one
    (d) a smoothed reconstruction (three pixel wide Gaussian convolution).\\
    It is apparent that (i) when there is too much noise, the divergence to the data does not converge to zero (ii) the percentile in the second column shows that when the noise is large enough, MLEM iterations quickly increase the sparsity.
  }
  \label{fig:noises}
\end{figure}

In \autoref{fig:noises}, we consider five different noise levels, associated to different values of dose $t$.
For each of these values, we are interested in seeing whether iterations lead to sparse measures or not.
From our results, this is equivalent to testing if $\data_t \in \imcone$.
A first crude estimate of this problem is to plot $d(\data_t||A \mu_k)$ and check whether this quantity converges to zero, which by theory implies $\data_t \in \imcone$. 
We also look at the \SI{95}{\percent} evolution of the percentile along the iterations.
A low percentile means that the mass concentrates on the remaining five percent of the image.

\subsection*{Dual certificates.}

It is difficult to make sure that the divergence to the data $\data$ converges to zero.
We thus also look for \emph{dual certificates} $\lambda$ in the dual cone $\imcone^*$ (see \autoref{sec:coneadjoint}) such that the dual function $g$ defined in~\eqref{dual_fct} fulfils $g(\lambda) > 0$.
Indeed, weak duality ensures that $\operatorname{min} d(\data_t|| A \mu) \geq g(\lambda)$, so the existence of any dual certificate proves that $\data_t \notin \imcone$.

In order to find a good choice for the certificate $\lambda$, we compute $\lambda_k = 1 -\frac{\data_t}{A \mu_k}$ along iterates,
which we know should converge to the optimal dual variable $\lambda^\star$ in the case where $\data_t \in \imcone$, and we conjecture this is true in full generality. Recall that the dual optimal variable is such that $A^*\lambda^\star \geq 0$, i.e., $\lambda^\star \in \imcone^*$.
For a fixed number of iterations $k$, there is no reason that $\lambda_k\in \imcone^*$, so we add a small appropriate constant $c$ to $\lambda_k$ and check that it provides a dual certificate, that is, we check whether $g(\lambda_k + c) > 0$. 

This procedure allows us to certify that in the three noisiest cases of \autoref{fig:noises}, the data $y_t$ is not in the cone $\imcone$.
We thus expect sparsity in those three cases.
We note that the resulting image after \num{400} iterates is not completely sparse: more iterates are required for only the Dirac masses to remain. 
We expect the convergence to the sum of Dirac masses to be very slow. 

In the two less noisy cases, we could not certify that  the data $\data_t$ is not in the cone $\imcone$, although that does not mean that the converse should be true.
In fact, it could well be that even for relatively low levels of noise, sparsity is the outcome but Dirac masses only take over after an practically unrealistic number of iterates.


\section{Open problems and perspectives} 

\subsection*{Convergence of iterates.} 
As shown in \autoref{thm_smooth}, cluster points of the~weak-$\ast$ cluster points of ML-EM are optimal when $y \in \imcone$.
A problem left open is their optimality when $y \notin \imcone$.
That result would imply a strong sparsity result for ML-EM cluster points.

Another problem is the convergence of the whole sequence to a single point, which is a tall order since there are in general many optimal points.
The discrete equivalent to~\autoref{decrease} allows to prove the full convergence of iterates in the discrete case, owing to the continuity of the discrete Kullback--Leibler divergence~\cite{Vardi1985} at cluster points.
In continuum, the fact that the divergence $D(\mu ||\nu)$ may be infinite even $\mu$ is absolutely continuous with respect to $\nu$ does not allow us to obtain a similar result, although we conjecture it does hold true.

\subsection*{Regularisation.}

Another interesting issue in light of our sparsity results is that of regularisation.
How should one choose appropriate additional regularisation terms to alleviate the problem? Similarly and as is done in~\cite{Resmerita2007} for continuous data, analysing the alternative strategy of regularising by early stopping is worthy of interest, since this is usual practice in PET. 

\subsection*{Going further in the case of PET}
For PET, the functions $a_i$ are actually close to being singular measures concentrated on a line. Studying the effect of this near-singularity on our results is of practical interest. More precisely, one could for specific geometries analyse the typical minimum set of a function of the form $A^* \lambda = \sum_{i=1}^{\ndet} \lambda_i a_i$.

It would also be natural to look at the effect of binning (i.e., aggregating detectors) on the constant $\inf_{q \in S \cap \imcone^c}  d(q,\data_r)$. Indeed,~\autoref{exp_dec} shows its importance when it comes to sparsity, justifying to try and make this distance as large as possible. 

\subsection*{Sparsity results in general.}
We intend to investigate the generality of these sparsity results in the context of other divergences. The squared-distance is a popular one, but generalisations have recently been advocated for in the literature, such as the $\beta$-divergences~\cite{Cavalcanti2019}. Finally, we believe our results can be extended without too much difficulty to the generalised statistical model for PET where scatter and random events are taking into account, namely
$y = \mathcal{P}(A \mu + s)$ with $s$ a known vector standing for the counts of scatter and random events.

\label{sec_conc}

\section*{Acknowledgements}
We are grateful to Sebastian Banert for fruitful discussions about optimisation in Banach spaces, as well as Axel Ringh and Johan Karlsson for bringing the moment matching problem to our attention. We acknowledge support from the Swedish Foundation of Strategic Research grant AM13-004.

%

\small
\bibliography{biblio.bib}
\bibliographystyle{acm}
\end{document}